\documentclass[10pt,a4paper]{amsart}
\usepackage[utf8]{inputenc}
\usepackage[T1]{fontenc}
\usepackage{amsmath}
\usepackage{amsfonts}
\usepackage{amssymb}
\usepackage{amsthm}
\numberwithin{equation}{section}
 \newtheorem{prop}[equation]{Proposition}
\newtheorem{lemma}[equation]{Lemma}
\newtheorem{corollary}[equation]{Corollary}
\newtheorem{definition}[equation]{Definition}
\newtheorem{theorem}[equation]{Theorem}
\newtheorem{example}[equation]{Example}
\newtheorem{observation}[equation]{\textbf{Observation}}
\begin{document}

\title{On the isotropy group of a simple derivation}

\author{L.N.Bertoncello}
\address{Universidade Federal de S\~ao Carlos, Via Washington Luiz, Km 235
	, 13565-905 S\~ao Carlos, Brazil}
\email{luciene@dm.ufscar.br}

\author{D. Levcovitz}
\address{Departamento de Matem\'atica, Instituto de Ci\^encias Matem\'aticas e de Computa\c{c}\~ao, Universidade de S\~ao Paulo, Caixa Postal 668, CEP 13560-970, S\~ao Carlos, SP, Brazil.}
\email{lev@icmc.usp.br}

\begin{abstract}
Let $R=K[X_1,\dots, X_n]$ be a polynomial ring in $n$ variables over a field $K$ of charactersitic zero and $d$ a $K$-derivation of $R$. Consider the isotropy  group if $d$: $ \text{Aut}(R)_d :=\{\rho \in \text{Aut}_K(R)|\; \rho d \rho^{-1}=d\}$. In his doctoral thesis (\cite{baltazar-2014}), Baltazar proved that if $d$ is a simple Shamsuddin derivation of $K[X_1,X_2]$, then its isotropy group is trivial. He also gave an example of a non-simple  derivation whose isotropy group is infinite. Recently, Mendes and Pan (\cite{mendes-pan}) generalized this result
to an arbitrary derivation of $K[X_1,X_2]$ proving that a derivation of $K[X_1,X_2]$ is simple if, and only if, its isotropy group is trivial.  In this paper, we prove that the isotropy group of a simple Shamsuddin derivation of the polynomial ring
$R=K[X_1,\dots, X_n]$ is trivial. We also calculate other isotropy groups of (not necessarily simple) derivations of $K[X_1,X_2]$ and prove that they are finite cyclic groups.

\end{abstract}

\maketitle

\section{Introduction}

Throughout this  paper, $K$ is a field of characteristic zero. Let $d$ be a derivation of a commutative $K$-algebra $R$. We say that $d$ is a {\it {simple derivation of $R$}}  (or just that $R$ is {\it $d$-simple}) if $R$ does not have  any proper non-zero ideal $I$ such that $d(I)\subseteq I$. Such an ideal $I$ is called a {\it $d$-invariant} ideal, a {\it $d$-stable} ideal or simply a {\it $d$-ideal}.

Research on simple derivations of commutative $K$-algebras has increased significantly in the past years. This was motivated, on one hand, by the  connection with the  theory of noncommutative noetherian simple rings. In fact, if $R[X;d]$ is the Ore extension of $R$ by $d$, then $R[X;d]$ is a simple ring (in the sense that it has no non-trivial two-sided ideals) if, and only if, $R$ is $d$-simple (see \cite{goodearl-warfield}). In this case, if $R$ is a noetherian $K$-algebra, then so is $R[X,d]$ (by Hilbert basis theorem) and $R[X;d]$ provides a useful example for testing conjectures in the (still mysterious) theory of  noncommutative noetherian simple rings.

On the other hand, there have been many recent connections of
$d-$ simplicity with commutative algebra and algebraic geometry via the theory of holomorphic foliations (see, for example \cite{collier-jalgebra2008}, \cite{papercollier-proceedings} and \cite{tesececilia}  ), $\mathcal{D}$-modules (see, for example, \cite{collier-geometriadedicata} and \cite{collier-glasgow2007}) and also with the question of algebraic independence of solutions of certain differential equations in the power series ring $K[[t]]$ (see \cite{brumatti-lequain-lev}).

We should mention, at least, two important results in commutative algebra that involve $d$-simple rings. A.Seidenberg showed in \cite{seidenberg1967} that if a finitely  generated domain $R$ admits a simple derivation $d$, then $R$ is regular. R.Hart showed in \cite{hart1975} that if $R$ is a localization of a finitely generated domain, then $R$ is regular if, and only if, $R$ admits a simple derivation $d$.

 Despite their ubiquity, simple derivations are far from being well  understood and even a characterization of simple derivations of the polynomial ring $K[X_1,\dots,X_n]$ cannot be foreseen up to now (the only known and trivial case is when $n=1$). Nevertheless, there is a class of polynomial derivations, the Shamsuddin derivations, whose simplicity  we  now understand better. A {\it Shamsuddin derivation} is a derivation $d$ of the polynomial ring $K[X; Y_1\dots,Y_n]$ of the type  $d=\partial_X+\sum_{i=1}^{n}(a_iY_i+b_i) \, \partial_{Y_i}$ with $a_i, b_i \in K[X]$ for every $i=1. \dots,n$. Owing to a seminal paper of Y.Lequain (\cite{lequain}), we can decide, effectively, if a Shamsuddin derivation is simple or not in terms of the polynomials $a_i$  and $b_i$. We will recall Lequain's result in the next paragraph.

Observe that simplicity is preserved by the action of the group of automorphisms of $R$, over the module of derivations of $R$, by conjugation. More precisely, if $R$ is a $K$-algebra, its $K$-automorphism group Aut$_K(R)$ acts by conjugation over the module of $K$-derivations of $R$, Der$_K(R)$: given $d\in$ Der$_K(R)$ and $\rho \in$ Aut$_K(R)$ then $\rho d \rho^{-1} \in$ Der$_K(R)$. Moreover, $d$ is simple if, and only if, $\rho d \rho^{-1}$ is  simple.

In order to understand  this action better, we focus on the stabilizer subgroup of the action. The {\it isotropy group} of  a derivation $d$ is its stabilizer subgroup:

$$ \text{Aut}(R)_d :=\{\rho \in \text{Aut}_K(R)|\; \rho d \rho^{-1}=d\}.$$
This group behaves well with respect to the conjugation action: if $d^{\prime}= \rho d \rho^{-1}$, then
$\text{Aut}(R)_{d^{\prime}}=\rho (\text{Aut}(R)_d) \rho^{-1}$; the isotropy groups are conjugated.

In his doctoral thesis, Baltazar (see \cite{baltazar-2014}) investigated the isotropy group of a simple Shamsuddin derivation of a polynomial ring in two variables
$K[X,Y]$. He showed that this group is trivial. He also calculated the isotropy group of the non-simple derivation $\partial_X$ of $K[X,Y]$ ((Y) is a $d$-ideal) and obtained the infinite group
$$\text{Aut} (K[X,Y])_{\partial_X} =\{\rho:(X,Y)\mapsto (X+p(Y),aY+b)|\, p(Y)\in K[Y], a\in K^{\star}, b\in K \}.$$ Baltazar thesis advisor was I. Pan. Based on this  result of his thesis and on this example, they conjectured:

{\bf Baltazar-Pan conjecture:} Let $d$ be a simple derivation of a finitely generated $K$-algebra $R$. Then its isotropy group is finite.

As far as we know, the status of this conjecture, up to now, is the following: as we mentioned, Baltazar proved it for  Shamsuddin derivations in two variables in his doctoral thesis (\cite{baltazar-2014}). Recently, Mendes and Pan proved the conjecture in dimension two (\cite{mendes-pan}). The main purpose of this paper is to prove the Baltazar-Pan conjecture for an arbitrary simple Shamsuddin derivation of a polynomial ring $K[X_1,\dots, X_n]$. In fact we show that, for this class of derivations, the isotropy group is trivial (see Theorem \ref{teorema1}). To do so, we strongly use the characterization of simple Shamsuddin derivations given by Lequain in \cite{lequain}.

This paper is organized as follows: in section 2 we recall the main results of \cite{lequain} that  will be used to prove our results. We also use this section  to establish the basic notations of the paper. In section 3, we prove the main theorem on the isotropy group of a simple Shamsuddin derivation (Theorem \ref{teorema1}). In section 4, we study the isotropy group of certain derivations of the polynomial ring in two variables  $K[X,Y]$ that may not be simple. We prove that their isotropy group are finite, but not necessarily trivial. In section 5, we consider derivations  that are simple, but not Shamsuddin. We claim, without proving it,  that their isotropy group are trivial as well. Finally, in section 6, based on all previous results and examples,  we formulate another conjecture on the isotropy group of a simple derivation of a polynomial ring.

\section{Simple Shamsuddin derivations: Lequain's characterization}

Recall that throughout this paper, $K$ is a field of characteristic zero. Given integers  $s,r_1,\dots,r_s \geq 1$, we consider $X\cup \{Y_{i,j};i=1,\dots,s,j=1,\dots,r_i\}$ a set of indeterminates over $K$. We denote the derivation $\partial_{Y_{i,j}}$ of the polynomial ring $K[X;\{Y_{i,j};i=1,\dots,s,\, j=1,\dots,r_i\}]$ simply by $\partial_{i,j}$. For an element $f\in K[X]$, we will often use $f^{\prime}$ instead of $\partial_X(f)$.

Remember that a derivation of the polynomial ring $K[X;Y_1, \cdots, Y_n]$ of the type $d=\partial_X+\sum_{i=1}^{n}(a_iY_i+b_i) \, \partial_{Y_i}$ with $a_i, b_i \in K[X]$, for every $i=1, \dots,n$, was defined to be a {\it Shamsuddin derivation}. Grouping the terms that have the same $a_i$, we can rewrite $d$ in the following form:
$$d=\partial_X + \sum_{i=1}^{s} \sum_{j=1}^{r_i}(a_iY_{i,j}+b_{i,j})\partial_{i,j} ,$$
with $a_i, b_{i,j}\in K[X]$ and $a_i\neq a_k$ if $i\neq k$. This specific form of the derivation $d$ is called the {\it canonical form} of $d$. The restriction $d|_{K[X;Y_{i,1},\dots,Y_{i,r_i}]}$, that is the derivation $d_i=\partial_X+ \sum_{j=1}^{r_i}(a_iY_{i,j}+b_{i,j})\partial_{i,j}$ of the polynomial ring  $K[X;Y_{i,1},\dots,Y_{i,r_i}]$ is called the i-th {\it canonical component} of $d$.

Although we will not need it explicitly, it would be important to begin \\ announcing the following simple and beautiful result  from Shamsuddin.

\begin{theorem} (Shamsuddin, \cite{shamsuddin}) Let $R$ be a ring, $Y$ an indeterminate over $R$ and $d$ a derivation of $R[Y]$ such that:
$d(R)\subseteq R$, $R$ is $d-$simple, $d(Y)=aY+b$ with $a,b\in R$. Then, the following statements are equivalent:
\begin{enumerate}
\item [(i)] $R[Y]$ is $d-$simple.
\item [(ii)] The equation $d(Z)=aZ+b$ does not have any solution in $R$.
\end{enumerate}
\end{theorem}

We will now recall  the main results of \cite{lequain}. They will be used abundantly in the next section.

\begin{theorem}\label{teoremayves3.1}(\cite{lequain}, Theorem 3.1, Local-Global Principle)
The following properties are equivalent:
\begin{enumerate}
\item[(i)] d is a simple derivation of $K[X;\{Y_{i,j};i=1,\dots,s,\, j=1,\dots,r_i\}]$.
\item[(ii)] For every $i=1,\dots,s$, $d_i=\partial_X+ \sum_{j=1}^{r_i}(a_iY_{i,j}+b_{i,j})\partial_{i,j}$ is a simple derivation of $K[X;Y_{i,1},\dots,Y_{i,r_i}]$.
\end{enumerate}
\end{theorem}

For this result to be useful, we need a criteria to decide if the i-th component of $d$, that is $d_i$,  is simple or not. This is given to  us by \cite{lequain}, Theorem 3.2. To enunciate it, we need  a Lemma and a definition.

\begin{lemma}[\cite{lequain}, Lemma 2.3]
Let $a,b \in K[X]$, $a\neq 0$. Consider the following sequences of equalities:

$\begin{array}{ccc}
b & = & a.q_1+r_1 \\
q_1^{\prime} & = & a.q_2+r_2\\
\vdots\\
q_t^{\prime} & = & a.0+r_{t+1},\\
\end{array}$
\vspace{3mm}

where $q_1, \dots,q_t,r_1,\dots,r_{t+1} \in K[X]$ and deg $r_i<$ deg $a$ for every $i$.

\begin{enumerate}
\item [(a)] The following statements are equivalent:
\subitem (i) The equation $Z^{\prime}=aZ+b$ has a solution in $K[X]$.
\subitem (ii) $\sum_{i=1}^{t+1} r_i=0$.
\item [(b)] If the equation  $Z^{\prime}=aZ+b$ has a solution $f \in K[X]$, then $f=-\sum_{i=1}^{t} q_i$.
\end{enumerate}
\end{lemma}

\begin{definition} (\cite{lequain}, Definition 2.4)
 Let $a,b\in K[X], a\neq 0$, be two polynomials and $r_1,\dots,r_{t+1}$ the sequence of polynomials
defined in the previous lemma. The polynomial $\sum_{i=1}^{t+1}r_i$ will be denoted by $\mathcal{P}(a,b)$.
\end{definition}

\begin{observation} \label{observation} Note that, in the definition above, if $a\in K^{\star}$, then $\mathcal{P}(a,b)=0$.
\end{observation}

\begin{theorem} \label{teoremayves3.2} (\cite{lequain}, Theorem 3.2) Let $i \in\{1,\dots,s\}.$
\begin{enumerate}
\item [(a)] The following properties are equivalent:
\subitem (i) $d_i$ is a simple derivation of $K[X; Y_{i,1},\dots, Y_{i,r_i}]$.
\subitem (ii) $a_i\neq 0$ and the polynomials $\mathcal{P}(a_i,b_{i,1}),\dots,\mathcal{P}(a_i,b_{i,r_{i}})$ are $K-$linear independent.
\subitem (iii) For every $(k_1,\dots,k_{r_i})\in K^{r_i}\setminus\{(0,\dots,0)\}$, the equation \\ $Z^{\prime}=a_iZ+\sum_{j=1}^{r_i}k_jb_{i,j}$
does not have any solution in $K[X]$.
\item [(b)] We can always determine effectively whether property $(ii)$ is satisfied or not.
\end{enumerate}

\end{theorem}

\section{The isotropy group of  simple Shamsuddin derivations}
Recall the following notation for the isotropy group of $d$: \newline Aut$(K[X;Y_1,\dots Y_n])_d=\{\rho \in \text{Aut}(K[X;Y_1,\dots Y_n])|\; \rho d \rho^{-1}=\text{id}\}.$
\begin{prop} \label{prop1}
Let  $d=\partial_X+\sum_{j=1}^{n}(a_iY_{i}+b_{i})\partial_{i}$ be a  simple Shamsuddin \\ derivation of the polynomial ring  $K[X;Y_1,\dots Y_n]$. If $\rho\in \text{Aut}(K[X; Y_1,\dots Y_n])_d $, then $\rho(X)= X+\alpha$, for some $\alpha \in K$.
\end{prop}

\begin{proof}[Proof]
Let us first show that $\rho(X)$ depends only on $X$. Write $\rho(X)=c_tY_i^t+\cdots + c_1Y_i+c_0$, for some $t \in \mathbb{N}$, $c_k\in K[X,Y_1, \dots , Y_{i-1}, Y_{i+1}, \dots , Y_n]$, $c_t\neq 0$.

Since $1=\rho (d(X))= d(\rho(X))$, then
$$
 1=d(\rho(X))=d(c_tY_i^t+\cdots + c_1Y_i+c_0)=(d(c_t)+ta_ic_t)Y^t_i+ \cdots$$
(terms of lower degree  in $Y_i$).

If $t \geq 1$, then $ d(c_t)+ta_ic_t=0$, that is, $(c_t)$ is a $d$-ideal. But this implies that $c_t=0$, because, as $d$ is simple, $a_i \neq 0$. Therefore $t=0$ and $\rho(X)$ does not depend on $Y_i$, for every $i$. Then $\rho(X)\in K[X]$.

But then, $1=\rho (d(X))= d(\rho(X))=\rho'(X)$. Thus  $\rho(X)=X+\alpha$ for some $\alpha \in K$.
\end{proof}

Let us write the polynomial ring $K[X;Y_1,\dots, Y_n]$ as
$K[X;\underline{Y_1},\dots,\underline{Y_s}]$ where
$\underline{Y_i}=\{Y_{i,1},\dots,Y_{i,r_i}\}$. Accordingly, the Shamsuddin  derivation $d$ has a canonical form $\partial_X+\sum_{i=1}^s\sum_{j=1}^{r_i}(a_iY_{i,j}+b_{i,j})\partial_{i,j}$, where $a_i,b_{i,j}\in K[X]$ and for every \\ $i\neq k, a_i\neq a_k$. To simplify even further the notation we denote the polynomial ring $K[X;\underline{Y_1},\dots,\underline{Y_s}]$ simply by $K[X;\{Y_{i,j}\}]$. Recall that each  derivation $d_i=\partial_X+ \sum_{j=1}^{r_i}(a_iY_{i,j}+b_{i,j})\partial_{i,j}$ of the polynomial ring  $K[X;Y_{i,1},\dots,Y_{i,r_i}]$ is called the i-th {\it canonical component} of $d$.

\begin{theorem} \label{teorema1}
If $d$ is a simple Shamsuddin derivation of the polynomial ring
$K[X;Y_1,\dots, Y_n], n\geq 1$, then its isotropy group is trivial.
\end{theorem}

\begin{proof}[Proof]
Let $Y$ and $W$ be any of the variables $Y_{i,j}$  for  $i=1,\dots,s$ and $ j=1,\dots,r_i$, including the case $W=Y$. We want to analyze the deg$_W \rho(Y)$. To simplify the notation
we put: $d(Y)=aY+b, d(W)=a'W+b'$ with $a,b,a',b' \in K[X]$.
We  write,
$$
 \rho(Y)=c_n W^n+c_{n-1}W^{n-1}+\cdots+c_0,$$
 with $c_k\in K[X,\{Y_{i,j}\}\setminus\{W\}], c_n\neq 0,$ so $n=$deg$_W \rho(Y)$.

 From $\rho(d(Y))=d(\rho(Y))$, we get $$\rho(a)\rho(Y)+\rho(b)=d(c_nW^n+c_{n-1}W^{n-1}+\cdots+c_0).$$

 Let us suppose that $n\geq 1$ and look for a contradiction, if $Y\neq W$. Comparing the coefficients of $W^n$ above and  taking into account that, by Proposition \ref{prop1},
  $\rho(a),\rho(b) \in K[X]$, and that $d(c_k)\in K[X,\{Y_{i,j}\}\setminus\{W\}]$, since $d$ is a Shamsuddin derivation, we have that
  $$
   \rho(a)c_n=d(c_n)+na'c_n \Rightarrow d(c_n)=(\rho(a)-na')c_n.
  $$

  As $d$ is simple, $c_n \in K^{\star}$ and  $\rho(a)=na'$. Changing the role of the variables $Y$ and $W$, we also have that
$\rho(a')=ma$, if $m:=$deg$_{Y}\rho(W) \geq 1$. Then, by Proposition \ref{prop1},
$$a'(X+\alpha) = \rho(a')= ma=m(\rho^{-1}(\rho(a)))=m(\rho^{-1}(na'))=mna'(X-\alpha) .$$
Since $a'\neq 0, m=n=1$. Then $a'(X+\alpha)=a'(X-\alpha)$. As
deg $a' \geq 1$, by Observation \ref{observation} and Theorem \ref{teoremayves3.2} (a) (ii), $\alpha=0$. But now $a'=\rho(a)=a$.

Note that we  just proved three things: $\rho(X)=X$; $Y$ and $W$ belong to the same component of $d$ (since $a=a'$) and deg$_W\rho(Y)=$ deg$_Y\rho(W)=1$.

Let us suppose, without loss of generality, that $Y$ and $W$ belong to the first component, say $Y,W \in \{Y_{1,1}, \dots, Y_{1,r_{1}}\}$. To simplify the notation we put
\\ $Y_1:= Y_{1,1}, \dots, Y_r:=Y_{1,r_{1}}$ and $d(Y_i)=aY_1+b_i$, $a,b_i\in K[X]$. Then, by the previous argument,
$$\rho(Y_1)=c_1Y_1+\dots+c_rY_r+c_0 $$ with $c_1,\dots, c_r \in K^{\star}$ and $c_0\in K[X]$.

If we impose the condition $d(\rho(Y_1))=\rho(d(Y_1)))$ again we have,
$$c_1d(Y_1)+ \cdots + c_r(Y_r)+ c_0^{\prime}=a(c_1Y_1+\cdots+c_rY_r+c_0)+b.  $$
Substituting $d(Y_i)=aY_i+b_i$ and canceling $ac_iY_i$ on both sides we get
\begin{equation}
c_0^{\prime}=ac_0-(c_1-1)b_1-\cdots -c_rb_r.
\label{equation1}
\end{equation}

If $r\geq 2$, this contradicts Theorem \ref{teoremayves3.2} (a)(iii), since $c_r\neq 0$. Therefore deg$_{Y_j}(Y_1)=0$ if $j=2,\dots,r$. That means that  $\rho(Y_1)$ does not depend on any of the variables $Y_2, \dots, Y_r$. Besides, $\rho(Y_1)=c_1Y_1+c_0, c_1\in K^{\star;}$ and $c_0\in K[X]$.

But now equation (\ref{equation1}) reads
$$c_0^{\prime}=ac_0-(c_1-1)b_1.$$ By Theorem \ref{teoremayves3.2} (a)(iii) again, $c_1=1$ and then $c_0=0$, since $a\neq 0.$
 Thus $\rho(Y_1)=Y_1$. As we already proved that $\rho(X)=X$, we have that $\rho=$ id.

\end{proof}

\section{Quadratic and cubic derivations}
 In this section, we study the isotropy group of some  derivations in two variables that are not  Shamsuddin derivations. Note that, in general, we will not suppose that the derivation is simple.

 \begin{definition}
  Let $K[X,Y]$ be the polynomial ring in two variables over a field $K$ of characteristic zero. A derivation $d$ of $K[X,Y]$ has $Y$-degree $n$ if $d(X)=1$ and $d(Y)$ has degree $n$ as a polynomial in $Y$ with coefficients in $K[X]$.
 \end{definition}

 Therefore, a derivation in two variables of $Y$-degree $n$ has the following form
 $$d=\partial_X+(h_0+h_1Y+\cdots + h_nY^n) \,\partial_Y, $$
  with $h_i \in K[X], h_n \neq 0$.

A derivations of $Y$-degree 2 (respectively 3) is called a {\it quadratic derivation in two variables} (respectively a {\it cubic derivation in two variables}).

{\bf Remark:} Our quadratic derivations differ a little  (and are more general)  from those studied by Maciejewski, Moulin-Ollagnier and Nowicki in \cite{nowickietal}. In their case, $d(Y)=Y^2+a(X)Y+b(X)$ is a monic polynomial in $Y$. They studied the simplicity of these derivations. Nowicki also proved that there are simple derivations in two variables with an arbitrarily large $Y$-degree (see \cite{nowicki}).

 \begin{theorem} \label{twovariables} Let $d$ be a derivation of the polynomial ring in two variables $K[X,Y]$ of $Y$-degree $n\geq 2$.  Let $\rho \in Aut(K[X,Y])_d$. Then,
  \begin{enumerate}
  \item [(i)]  $\rho(X)=X$ and $\rho(Y)=b_0+b_1Y$ with $b_0\in K[X], b_1\in K^{\star}$  and $b_1$ satisfies \\ $b^{n-1}_1=1$.
  \item[(ii)] If $b_1=1$, then $b_0=0$. In this case $\rho=$id.
  \item[(iii)] Let $d(Y)=h(X,Y)=h_0+h_1Y+\cdots + h_nY^n$ with $h_i \in K[X], h_n \neq 0$. If $h_0\neq 0$ and $b_0=0$, then $b_1=1$. In this case $\rho=$id.
  \item[(iv)] If $d$ is simple and $b_0=0$, then $b_1=1$. In this case $\rho=$id.
  \end{enumerate}
  \end{theorem}

 \begin{proof}[Proof]
 \begin{enumerate}

\item[(i)] Let us write  $\rho(X)=a_0+a_1Y+\cdots +a_sY^s$ with $a_i \in K[X], a_s \neq 0$. We will show first that $s=0$.

 Since $d(X)=1$ and $\rho(d(X))=d(\rho(X))$ we have,
 $$1=d(a_0+a_1Y+\cdots +a_sY^s) =$$
 $$a_0^{\prime} +a_1^{\prime}Y+\cdots +a_s^{\prime}Y^s +a_1d(Y)+\cdots +sa_sY^{s-1}d(Y).$$
 Substituting $d(Y)$ for $h_0+\cdots+ h_nY^n$ and collecting the coefficient of $Y^{n+s-1}$ on both sides of the above equation we get $sa_sh_n=0$, since  $n+s-1\geq s+1\geq 1$. As $a_sh_n \neq 0$, then $s=0$.

 Therefore $\rho(X)\in K[X]$. Imposing the condition $\rho(d(X))=d(\rho(X))$ again we have $1=\rho^{\prime}(X)$. Then  $\rho(X)=X+\alpha$, for some $\alpha \in K$.

 Let us write now $\rho(Y)=b_0+b_1Y+\cdots +b_tY^t$
 with $b_i \in K[X], b_t \neq 0$. Then $\rho(b_i(X))=b_i(X+\alpha)$. Thus,
 $$\rho(d(Y))=h_0(X+\alpha)+h_1(X+\alpha)(b_0+\cdots +b_tY^t)+\cdots +h_n(X+\alpha)(b_0+\cdots +b_tY^t)^n. $$
 On the other hand,
 $$d(\rho(Y))=b_0^{\prime} +b_1^{\prime}Y+\cdots +b_t^{\prime}Y^t +b_1d(Y)+\cdots +tb_tY^{t-1}d(Y). $$
Since $\rho(d(Y))=d(\rho(Y))$, substituting $d(Y)$ for $h_0+\cdots+ h_nY^n$ and  \\ comparing the terms of higher degree in $Y$ we get
$$h_n(X+\alpha)b_t^nY^{nt}=tb_th_n(X)Y^{n+t-1} .$$
Then $t=1, b_t^{n-1}=1$ and $\alpha=0$.

Therefore $\rho(X)=X$ and $\rho(Y)=b_0+b_1Y$ with $b_0\in K[X], b_1\in K^{\star}$ and $b_1$ satisfies $b^{n-1}_1=1$.

\item[(ii):] We impose again the condition $\rho(d(Y))=d(\rho(Y))$, but now, by (i) and the hypothesis, $\rho(Y)=b_0+Y$. Looking at the coefficient of degree $n-1$ in $Y$ we have $nh_nb_0=0$. This implies that $b_0=0$, since $nh_n\neq 0$. Then $\rho=$id.

\item[(iii):] Now we look at the constant term in $\rho(d(Y))=d(\rho(Y))$. We have
\begin{equation} \label{eq1}
h_nb_0^{n}+h_{n-1}b_0^{n-1}+\cdots +h_1b_0+h_0(1-b_1)=b^{\prime}_0.
\end{equation}
If  $h_0\neq 0$ and $b_0=0$, then $b_1=1$ and  $\rho=$id.

\item[(iv):] If $d$ is simple, then $h_0\neq 0$, otherwise $(Y)$ would be a $d$-ideal. Now it follows from (iii) above.
\end{enumerate}
\end{proof}

\begin{corollary} \label{corollarycyclic}
Let $d$ be a derivation in two variables of $Y$-degree $n\geq 2$. Let $\mu_{n-1}(K)$ denote the cyclic group of $n-1$ roots of unity in $K$. Then  the isotropy group of $d$ is a subgroup of $\mu_{n-1}(K)$. In particular, it is a finite cyclic group.
\end{corollary}
\begin{proof}[Proof]
Consider the map $\varphi:K[X,Y]_d \rightarrow \mu_{n-1}(K)$ given by $\varphi(\rho)=b_1$ where $\rho(Y)=b_0+b_1Y$. It is a group homomorphisms. By theorem \ref{twovariables} (i), it is well defined; by (ii) it is injective. Then the result follows.
\end{proof}

\begin{corollary} Let $d$ be a derivation in two variables of $Y$-degree $n\geq 2$ over a field $K$. Suppose that $K \subseteq \mathbb{R}$. Then the isotropy group of $d$ is either trivial or cyclic with two elements.
\end{corollary}
\begin{proof}[Proof]
Since $K \subseteq \mathbb{R}$, the only possible roots of units in $K$ are $\{\pm 1\}$.
\end{proof}

An  interesting consequence of  Corollary \ref{corollarycyclic} is that a quadratic derivation has a trivial isotropy group, regardless of whether it is simple or not. But cubic derivations can have an isotropy group of order 2.

\begin{corollary}
\begin{enumerate}

\item [(i)]If $d$ is a quadratic derivation  in two variables, then its isotropy group is trivial.
\item [(ii)]If $d$ be a cubic derivation  in two variables, then its isotropy group is either trivial or a group of order 2 (and both cases occur).
\end{enumerate}
\end{corollary}
\begin{proof}[Proof]
It follows immediately from Corollary \ref{corollarycyclic}.
\end{proof}

\begin{example}
Let $d$ be the cubic derivation in two variables given by \\ $d(X)=1$, $d(Y)=h_1+h_3Y^3$,  $h_1, h_3\in K[X], h_3\neq 0$. Let $\rho$ be the isomorphisms of $K[X,Y]$ given by $\rho(X)=X$, $\rho(Y)=-Y$. Then $K[X,Y]_d=\{\text{id}, \rho\}$. It is therefore a cyclic group of order 2.
\end{example}

\begin{example}
Let $d$ be the derivation in two variables given by $d(X)=1$, \\ $d(Y)=Y^n+ pX$, where $n\geq 2,\, p\in K^{\star}$. Nowicki proved in \cite{nowicki} that $d$ is simple. We claim that its isotropy group is trivial. In fact, let $\rho\in K[X,Y]_d$. By theorem \ref{twovariables}, $\rho(Y)=b_0+b_1Y$, $b_0\in K[X], b_1\in K^{\star}$. By equation \ref{eq1}, we have $b_0^n+pX(1-b_1)=b_0^{\prime}$. Since $n\geq 2$,  $b_0\in K$. Then $p(1-b_1)X+b_0^n=0$. Then $b_1=1, b_0=0$ and therefore $\rho=$id.
\end{example}

 \section{Examples}

In this section, we give  three more examples. The first is of a non-simple \\ Shamsuddin derivation with an infinite isotropy group. Actually, Baltazar, in his thesis (\cite{baltazar-2014}), had already given such an example. The difference here is that our example is a Shamsuddin derivation in three variables while his example is not a Shamsuddin derivation and it is in two variables.

\begin{example} Let $d$ be the derivation of the polynomial ring $K[X,Y,Z]$
given by $d(X)=1, d(Y)=(1+XY)$ and $d(Z)=1+XZ$. This is a non-simple Shamsuddin derivation of $K[X,Y,Z]$, since the ideal $(Y-Z)$ is a $d$-ideal.
Its isotropy group is given by $\rho \in$ Aut$(K[X,Y,Z])$ such that $\rho(X)=X, \rho(Y)= aY+bZ, \rho(Z)=cY+dZ$ with $a+b=1, c+d=1$ and $ad-bc\neq 0$. Therefore it is an infinite group.
\end{example}

We now give  two examples of simple derivations of polynomial rings that are not Shamsuddin derivations but still have a trivial isotropy group. Since the techniques of proofs are very similar to the ones we have been given before, we will omit them.

\begin{example} Consider the derivation of the polynomial ring $K[X,Y,Z]$ given by  $d=\partial_X + (XY+1)\partial_Y + (YZ+1)\partial_Z$.  Nowicki proved in  \cite{nowickibook}, Example 13.4.3 that  this derivation is simple. Its isotropy group is trivial. Note that $d$ is not a Shamsuddin derivation.
\end{example}

\begin{example} Consider the derivation of the polynomial ring in $n$ variables \\ $K[X_1,X_2,\dots,X_n]$ given by  $d(X_1)=1, d(X_i)=a_iX_i+b_i$, where $a_i,b_i\in K[X_1,\dots,X_{i-1}], b_i\neq 0, deg_{X_{i-1}}(a_i)\geq 1$ and
$deg_{X_{i-1}}(b_i)< deg_{X_{i-1}}(a_i)$.
 Nowicki proved in \cite{nowickibook}, Example 13.4.2 that  this derivation is simple. Its isotropy group is trivial. Note that $d$ is not a Shamsuddin derivation.
\end{example}

\section{A new conjecture on the isotropy group of a simple derivation.}
Based on Baltazar (\cite{baltazar-2014},\cite{baltazar-2016})  and  Mendes-Pan results (\cite{mendes-pan}) and also on the results and examples we presented in this paper, we can make the following conjecture:

\textbf{Conjecture:} Let $d$ be a  derivation of the polynomial ring $K[X_1,X_2,\dots,X_n]$ in $n$ variables over a field $K$ of characteristic zero. Then $d$ is simple if, and only if, its isotropy group is trivial.

 \end{document}